\documentclass[a4paper]{amsart}

\usepackage{amsmath,amssymb,amsthm}
\usepackage[all]{xy}
\usepackage{eucal}
\usepackage{graphicx}
\usepackage{tikz}
\usepackage[colorlinks=true,backref=page]{hyperref}

\newtheorem{theorem}{Theorem}[section]
\newtheorem{proposition}[theorem]{Proposition}
\newtheorem{lemma}[theorem]{Lemma}
\newtheorem{corollary}[theorem]{Corollary}

\theoremstyle{definition}
\newtheorem{definition}[theorem]{Definition}

\newtheorem{problem}[theorem]{Problem}
\newtheorem{question}[theorem]{Question}

\theoremstyle{remark}
\newtheorem{remark}[theorem]{Remark}

\numberwithin{equation}{section}

\newcommand{\Z}{\mathbb{Z}}

\newcommand{\R}{\mathbb{R}}

\newcommand{\Conf}{\mathtt{Conf}}

\newcommand{\st}{\operatorname{st}}
\newcommand{\wgt}{\operatorname{wgt}}
\newcommand{\hocolim}{\operatorname{hocolim}}

\SelectTips{cm}{}

\title[Tverberg's theorem for cell complexes]{Tverberg's theorem for cell complexes}

\author[S. Hasui]{Sho Hasui}
\address{Department of Mathematical Sciences, Osaka Prefecture University, Sakai, 599-8531, Japan}
\email{s.hasui@ms.osakafu-u.ac.jp}

\author[D. Kishimoto]{Daisuke Kishimoto}
\address{Department of Mathematics, Kyoto University, Kyoto, 606-8502, Japan}
\email{kishi@math.kyoto-u.ac.jp}

\author[M. Takeda]{Masahiro Takeda}
\address{Department of Mathematics, Kyoto University, Kyoto, 606-8502, Japan}
\email{takeda.masahiro.87u@st.kyoto-u.ac.jp}

\author{M. Tsutaya}
\address{Faculty of Mathematics, Kyushu University, Fukuoka 819-0395, Japan}
\email{tsutaya@math.kyushu-u.ac.jp}

\date{\today}

\subjclass[2010]{52A37, 55R80}

\keywords{topological Tverberg theorem, complementary acyclicity, simplicial sphere, discretized configuration space, homotopy colimit}

\begin{document}

\maketitle

\begin{abstract}
  The topological Tverberg theorem states that any continuous map of a $(d+1)(r-1)$-simplex into the Euclidean $d$-space maps some points from $r$ pairwise disjoint faces of the simplex to the same point whenever $r$ is a prime power. We substantially generalize this theorem to continuous maps of certain CW complexes, including simplicial $((d+1)(r-1)-1)$-spheres, into the Euclidean $d$-space. We also discuss the atomicity of the Tverberg property.
\end{abstract}


\section{Introduction}\label{introduction}

Tverberg's theorem \cite{T} states that any given $(d+1)(r-1)+1$ points in $\R^d$ can be a partitioned into $r$ disjoint subsets such that the convex hulls of these subsets have a point in common. This has been of great interest in combinatorics for over 50 years, and a variety of its generalization have been obtained. See comprehensive surveys \cite{BBZ,BS,BZ} for history and developments around Tverberg's theorem. Among others, a topological generalization of Tverberg's theorem was conjectures by B\'{a}r\'{a}ny in 1976, and was affirmatively solved as follows, which is now called the topological Tverberg theorem. Let $\Delta^n$ denote an $n$-simplex.

\begin{theorem}
  \label{TTT}
  If $r$ is a prime power, then for any continuous map $f\colon\Delta^{(d+1)(r-1)}\to\R^d$, there are pairwise disjoint faces $\sigma_1,\ldots,\sigma_r$ of $\Delta^{(d+1)(r-1)}$ such that $f(\sigma_1),\ldots,$
  $f(\sigma_r)$ have a point in common
\end{theorem}

The topological Tverberg theorem was proved by B\'{a}r\'{a}ny, Shlosman and Sz\H{u}cs \cite{BSS} when $r$ is a prime, and by \"{O}zaydin \cite{O} and Volovoikov \cite{V1} when $r$ is a prime power. When $r$ is a prime, the proof employs a generalization of the Borsuk-Ulam theorem, and when $r$ is a prime power, Volovikov \cite{V1} employed a cohomological index, which is essentially the same as the ideal-valued cohomological index of Fadell and Husseini \cite{FH}. Remarkably, Frick \cite{F} proved that the condition that $r$ is a prime power is necessary.

Let us consider a generalization of the topological Tverberg theorem. In \cite{GS}, Tverberg asked whether or not it is possible to generalize the topological Tverberg theorem to continuous maps from $(d+1)(r-1)$-polytopes into $\R^d$. The answer is positive because the boundary of a convex $n$-polytope is a refinement of the boundary of an $n$-simplex as in \cite[p. 200]{G} and the result follows from the topological Tverberg theorem. Then Tverberg's question does not contribute to a proper generalization of the topological Tverberg theorem, and so we further ask:

\begin{question}
  \label{question}
  For which CW complexes can we generalize the topological Tverberg theorem to continuous maps from them into Euclidean spaces?
\end{question}

Recently, B\'{a}r\'{a}ny, Kalai and Meshulam \cite{BKM} and Blagojevi\'{c}, Haase and Ziegler \cite{BHZ} constructed affirmative examples of matroid complexes for Question \ref{question} in a purely combinatorial way. In this paper, we give a new affirmative class of regular CW complexes from a topological point of view, which will turn out to be substantial. To state the main theorem, we set notation and terminology. Let $X$ be a regular CW complex. A \emph{face} of $X$ means its closed cell. For faces $\sigma_1,\ldots,\sigma_k$ of $X$, let $X(\sigma_1,\ldots,\sigma_k)$ denote the subcomplex of $X$ consisting of faces which do not intersect with $\sigma_1,\ldots,\sigma_k$. Recall that a space $Y$ is called  \emph{$n$-acyclic} if $\widetilde{H}_*(Y)=0$ for $*\le n$. For convenience, a non-empty space will be called $(-1)$-acyclic, so that any $n$-acyclic space for $n\ge 0$ will be assumed non-empty. We define a regular CW complex that we are going to consider in this paper.

\begin{definition}
  We say that a regular CW complex $X$ is \emph{$k$-complementary $n$-acyclic} if $X(\sigma_1,\ldots,\sigma_i)$ is $(n-\dim\sigma_1-\cdots-\dim\sigma_i)$-acyclic for any pairwise disjoint faces $\sigma_1,\ldots,\sigma_i$ of $X$ such that $\dim\sigma_1+\cdots+\dim\sigma_i\le n+1$ and $0\le i\le k$.
\end{definition}

Now we state the main theorem.

\begin{theorem}
  \label{main}
  If $X$ is an $(r-1)$-complementary $(d(r-1)-1)$-acyclic regular CW complex and $r$ is a prime power, then for any continuous map $f\colon X\to\R^d$, there are pairwise disjoint faces $\sigma_1,\ldots,\sigma_r$ of $X$ such that $f(\sigma_1),\ldots,f(\sigma_r)$ have a point in common.
\end{theorem}

Since a $(d+1)$-simplex is $k$-complementary $(d-k)$-acyclic for $1\le k\le d+1$, the topological Tverberg theorem is recovered by Theorem \ref{main}. Moreover, we will prove that every simplicial $d$-sphere is $k$-complementary $(d-k)$-acyclic for $1\le k\le d+1$ (Proposition \ref{sphere}). Then we get:

\begin{corollary}
  \label{Tverberg sphere}
  If $S$ is a simplicial $((d+1)(r-1)-1)$-sphere and $r$ is a prime power, then for any continuous map $f\colon S\to\R^d$, there are pairwise disjoint faces $\sigma_1,\ldots,\sigma_r$ of $S$ such that $f(\sigma_1),\ldots,f(\sigma_r)$ have a point in common.
\end{corollary}

Let us go back to Tverberg's question mentioned above. As we saw above, the topological Tverberg theorem is generalized to maps out of polytopal spheres, where a polytopal sphere means the boundary of a convex polytope. However, a generalization to simplicial spheres has not been proved. Gr\"{u}nbaum and Sreedharan \cite{GSr} constructed a simplicial 3-sphere which is not polytopal. Moreover, Kalai \cite{K} proved that for $d$ large,  "most" simplicial $d$-spheres are not polytopal. Then Corollary \ref{Tverberg sphere}, hence Theorem \ref{main}, is a substantial generalization of the topological Tverberg theorem. On the other hand, (the failure of) Tverberg's question motivates us to consider the atomicity of the Tverberg property, that is, the Tverberg property which is not induced from subcomplexes or refinement. In Section \ref{atomicity}, we will pose a counting problem of atomic complexes having the Tverberg property, and we will give some computations in the special cases.

We sketch the outline of the proof of Theorem \ref{main}. We will apply the standard topological method in combinatorics to prove Theorem \ref{main} by using discretized configuration spaces, or deleted products, and the weight of a group action, which is essentially the same as Fadell and Husseini's ideal valued cohomological index \cite{FH}. Then the proof reduces to computing acyclicity of discretized configuration spaces. We will give a description of discretized configuration spaces in terms of homotopy colimits (Theorem \ref{hocolim}), which can be thought of as a combinatorial version of the Fadell-Neuwirth fibration \cite{FN} and has other applications such as \cite{KM}. This enables us to compute the acyclicity by the Bousfield-Kan type spectral sequence, which actually leads us to the complementary acyclicity.


\subsection*{Acknowledgement}

The authors were supported by JSPS KAKENHI Grant Numbers JP18K13414 (Hasui), JP17K05248 and JP19K03473 (Kishimoto), JP21J10117 (Takeda), and JP19K14535 (Tsutaya).


\section{Simplicial sphere}\label{simplicial sphere}

This section proves that every simplicial $d$-sphere is $k$-complementary $(d-k)$-acyclic for $1\le k\le d+1$. We will use the nerve theorem, so we set notation and terminology for it. Let $\mathcal{U}$ be an open cover of a topological space $X$. We say that $\mathcal{U}$ is \emph{good} if each intersection of finitely many members of $\mathcal{U}$ is either empty or contractible. Let $N(\mathcal{U})$ denote the nerve of $\mathcal{U}$, which is a simplicial complex such that vertices correspond to members of $\mathcal{U}$ and finitely many vertices form a simplex if the intersection of the corresponding members of $\mathcal{U}$ is non-empty. Now we state the nerve theorem. See \cite[Corollary 4G.3]{Ha} for the proof.

\begin{lemma}
  \label{nerve}
  If $\mathcal{U}$ is a good open cover of a paracompact space $X$, then
  \[
    X\simeq N(\mathcal{U}).
  \]
\end{lemma}

Let $K$ be a simplicial complex. For a simplex $\sigma$ of $K$, let $\st(\sigma)$ denote the open star of $\sigma$. Namely, $\st(\sigma)$ is the union of the interiors of all simplices including $\sigma$. Note that every open star is contractible.

\begin{lemma}
  \label{cover}
  Let $K$ be a simplicial complex. For vertices $v_1,\ldots,v_n$ of $K$,
  \[
    \widetilde{H}^*(\st(v_1)\cup\cdots\cup\st(v_n))=0\quad(*\ge n-1).
  \]
\end{lemma}

\begin{proof}
  Let $U=\st(v_1)\cup\cdots\cup\st(v_n)$. For $n=1$, $U=\st(v_1)$ is contractible, so the statement is obvious. Then we may assume $n\ge 2$, and it is sufficient to show that the homotopy dimension of $U$ is at most $n-2$. We consider an open cover $\mathcal{U}=\{\st(v_i)\}_{i=1}^n$ of $U$. As in \cite[p. 372]{M}, if $\st(v_{i_1})\cap\cdots\cap\st(v_{i_k})\ne\emptyset$, then vertices $v_{i_1},\ldots,v_{i_k}$ form a simplex $\sigma$ of $K$ such that $\st(v_{i_1})\cap\cdots\cap\st(v_{i_k})=\st(\sigma)$, which is contractible. Thus $\mathcal{U}$ is a good open cover of $U$, and by Lemma \ref{nerve}, we get
  \[
    U\simeq N(\mathcal{U}).
  \]
  By definition, $N(\mathcal{U})$ is a simplicial complex with $n$ vertices, so $\dim N(\mathcal{U})\le n-1$. If $\dim N(\mathcal{U})=n-1$, $N(\mathcal{U})$ is a simplex which is contractible. Thus we obtain that the homotopy dimension of $N(\mathcal{U})$ is at most $n-2$, completing the proof.
\end{proof}

We are ready to prove:

\begin{proposition}
  \label{sphere}
  Every simplicial $d$-sphere is $k$-complementary $(d-k)$-acyclic for $1\le k\le d+1$.
\end{proposition}

\begin{proof}
  Let $S$ be a simplicial $d$-sphere, and let $\sigma_1,\ldots,\sigma_i$ be simplices of $S$ such that $\dim\sigma_1+\cdots+\dim\sigma_i\le d-k+1$ and $0\le i\le k$. Let $v_1,\ldots,v_m$ be vertices of $\sigma_1,\ldots,\sigma_i$. Then $m\le i+\dim\sigma_1+\cdots+\dim\sigma_i$. Since $S$ is a simplicial $d$-sphere, it has at least $d+2$ vertices. On the other hand, we have $m\le i+\dim\sigma_1+\cdots+\dim\sigma_i\le d-k+i+1\le d+1$. Then $S(\sigma_1,\ldots,\sigma_i)$ is non-empty. It remains to compute the homology of $S(\sigma_1,\ldots,\sigma_i)$. Let $U=\st(v_1)\cup\cdots\cup\st(v_m)$. Then we have
  \[
    S(\sigma_1,\ldots,\sigma_i)=S-U.
  \]
  By shrinking $U$ slightly, we get a closed and locally contractible subset $C$ of $U$ such that $C\simeq U$ and $S-C\simeq S-U$. By the Alexander duality, we have
  \[
    \widetilde{H}_{d-j-1}(S-C)\cong\widetilde{H}^j(C).
  \]
  By Lemma \ref{cover}, $\widetilde{H}^j(C)\cong\widetilde{H}^j(U)=0$ for $j\ge m-1$, and so
  \[
    \widetilde{H}_*(S(\sigma_1,\ldots,\sigma_i))=\widetilde{H}_*(S-U)\cong\widetilde{H}_*(S-C)=0
  \]
  for $*\le d-m$. Thus since $d-k-\dim\sigma_1-\cdots-\dim\sigma_i\le d-i-\dim\sigma_1-\cdots-\dim\sigma_i\le d-m$, the proof is finished.
\end{proof}

It is quite probable that every polyhedral $d$-sphere is $k$-complementary $(d-k)$-acyclic for $1\le k\le d+1$. However, complementary acyclicity of regular CW spheres seem more complicated. The minimal regular CW decomposition of a $d$-sphere $S^d=e_-^0\cup e_+^0\cup\cdots\cup e^d_-\cup e^d_+$ for $d\ge 1$ is not 1-complementary 0-acyclic because $S^d$ minus a 1-face is empty. On the other hand, there are non-polyhedral $d$-spheres which are $k$-complementary $(d-k)$-sphere for $1\le k\le d+1$. For example, the regular CW 2-sphere depicted below is $k$-complementary $(2-k)$-acyclic for $1\le k\le 3$. It would be interesting to find a condition for a regular CW $d$-sphere being $k$-complementary $(d-k)$-acyclic for $1\le k\le d+1$.

\begin{figure}[htbp]
  \centering
  \vspace{2mm}
  \includegraphics[width=3.2cm]{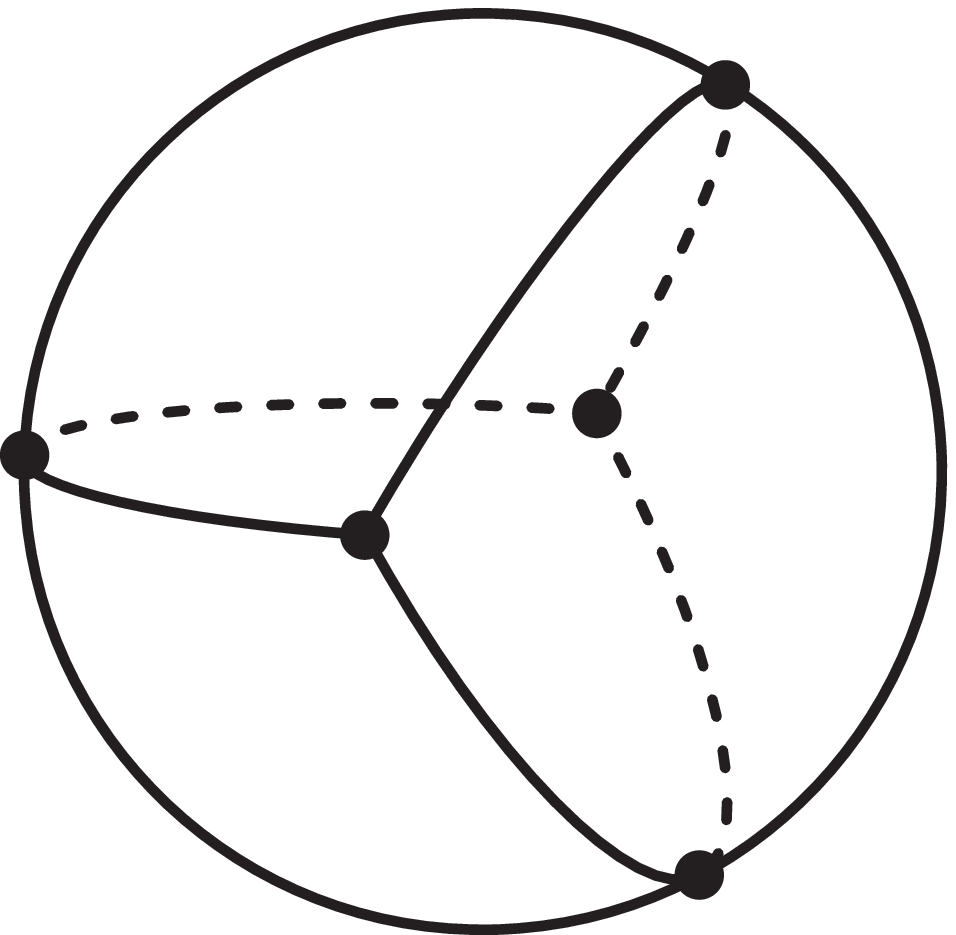}
  \caption{}
  \label{non-polyhedral}
\end{figure}


\section{Group action}\label{group action}

This section connects Theorem \ref{main} to equivariant topology by the standard topological method in combinatorics. Let $X$ be a regular CW complex. The \emph{discretized configuration space} $\Conf_r(X)$ is defined as the subcomplex of the direct product $X^r$ consisting of faces $\sigma_1\times\cdots\times\sigma_r$ such that $\sigma_1,\ldots,\sigma_r$ are pairwise disjoint faces of $X$. The discretized configuration space is often called the deleted product in combinatorics, alternatively. Let $\Delta=\{(x_1,\ldots,x_r)\in(\R^d)^r\mid x_1=\cdots=x_r\}$. There is a homotopy equivalence
\begin{equation}
  \label{R-D}
  (\R^d)^r-\Delta\simeq S^{d(r-1)-1}.
\end{equation}
Note that the symmetric group $\Sigma_r$ acts on $\Conf_r(X)$ and $(\R^d)^r-\Delta$ by permuting of entries. The following lemma is proved in \cite[Theorem 3.9]{BZ}.

\begin{lemma}
  \label{equivariant map}
  Let $X$ be a regular CW complex. If there is a continuous map $X\to\R^d$ such that $f(\sigma_1)\cap\cdots\cap f(\sigma_r)=\emptyset$ for all pairwise disjoint faces $\sigma_1,\ldots,\sigma_r$ of $X$, then there is a $\Sigma_r$-map
  \[
    \Conf_r(X)\to(\R^d)^r-\Delta.
  \]
\end{lemma}

To apply Lemma \ref{equivariant map}, we will use the following invariant. Let $G=(\Z/p)^k$, and let $X$ be a $G$-space. The invariant $\wgt_G(X)$ is defined to be the greatest integer $n$ such that the natural map
\[
  H^n(BG;\Z/p)\to H^n(EG\times_GX;\Z/p)
\]
is injective. We call $\wgt_G(X)$ the \emph{weight} from the view of the category weight \cite{R}, instead of the over used "index" as in \cite{V1}. It is easy to see that $\wgt_G(X)$ is essentially the same as the ideal-valued cohomological index for $EG\times_GX$ due to Fadell and Husseini \cite{FH}. We will use the following properties of weights, which were also used in \cite{BZ,V1,V2}.

\begin{lemma}
  \label{weight}
  For $G=(\Z/p)^k$, the following statements hold.
  \begin{enumerate}
    \item If there is a $G$-map $X\to Y$ between $G$-spaces $X,Y$, then
    \[
      \wgt_G(X)\le\wgt_G(Y).
    \]

    \item If a $G$-space $X$ is $n$-acyclic, then
    \[
      \wgt_G(X)\ge n+1.
    \]

    \item If a paracompact space $S$ satisfies $H^*(S)\cong H^*(S^n)$ and $G$ acts on $S$ without fixed point, then
    \[
      \wgt_G(S)=n.
    \]
  \end{enumerate}
\end{lemma}

\begin{proof}
  (1) This is obvious from the definition of weights.

  \noindent(2) Consider the mod $p$ cohomology Serre spectral sequence for a fibration $X\to EG\times_GX\to BG$. Then since $X$ is $n$-acyclic, $E_2^{0,q}=0$ for $1\le q\le n$. Thus the map $H^*(BG;\Z/p)\to H^*(EG\times_GX;\Z/p)$ must be injective for $*\le n+1$, implying $\wgt_G(X)\ge n+1$.

  \noindent(3) By (2), we only need to prove $\wgt_G(S)\le n$. Since the $G$-action on $S$ is fixed point free, the map $H^*(BG;\Z/p)\to H^*(EG\times_GS;\Z/p)$ is not injective by \cite[Corollary 1, Chapter IV]{H}. Consider the mod $p$ cohomology Serre spectral sequence for a fibration $S\to EG\times_GS\to BG$. Since the group of automorphisms of $\Z/p$ is isomorphic to $\Z/(p-1)$ and there is no non-trivial homomorphism $G\to\Z/(p-1)$, the local coefficients in the spectral sequence is trivial. Then since the map $H^*(BG;\Z/p)\to H^*(EG\times_GS;\Z/p)$ is not injective, the transgression $E_n^{0,n}\to E_2^{n+1,0}$ must be non-trivial. Thus we obtain that the map $H^{n+1}(BG;\Z/p)\to H^{n+1}(EG\times_GS;\Z/p)$ is not injective, implying $\wgt_G(S)\le n$. Thus the proof is complete.
\end{proof}

Now we apply weights to Lemma \ref{equivariant map}. Let $r=p^k$ and $G=(\Z/p)^k$. The map
\[
  G\times G\to G,\quad(x,y)\mapsto x+y
\]
defines a faithful $G$-action on $G$ itself, so that we get a monomorphism $G\to\Sigma_r$. Then we can consider the induced $G$-action on a $\Sigma_r$-space through this monomorphism. We are ready to prove:

\begin{proposition}
  \label{criterion}
  Let $X$ be a regular CW complex such that $\Conf_r(X)$ is $(d(r-1)-1)$-acyclic. If $r$ is a prime power, then for any continuous map $f\colon X\to\R^d$, there are pairwise disjoint faces $\sigma_1,\ldots,\sigma_r$ of $X$ such that $f(\sigma_1),\ldots,f(\sigma_r)$ have a point in common.
\end{proposition}

\begin{proof}
  Let $r=p^k$ and $G=(\Z/p)^k$. Suppose that all pairwise disjoint faces $\sigma_1,\ldots,\sigma_r$ of $X$ satisfy $f(\sigma_1)\cap\cdots\cap f(\sigma_r)=\emptyset$. Then by Lemma \ref{equivariant map}, there is a $G$-map $\Conf_r(X)\to(\R^d)^r-\Delta$, and so by Lemma \ref{weight}, we get
  \[
    d(r-1)\le\wgt_G(\Conf_r(X))\le\wgt_G((\R^d)^r-\Delta).
  \]
  On the other hand, by \eqref{R-D} and Lemma \ref{weight},
  \[
    \wgt_G((\R^d)^r-\Delta)\le d(r-1)-1.
  \]
  Thus we obtain a contradiction, finishing the proof.
\end{proof}


\section{homotopy colimit}\label{homotopy colimit}

This section describes a discrete configuration space in terms of a homotopy colimit, and proves Theorem \ref{main}. We recall from \cite{ZZ} the definition of a homotopy colimit of a functor from a poset. Let $P$ be a poset. Hereafter, we understand $P$ as a category such that objects are elements of $P$ and there is a unique morphism $x\to y$ for $x>y\in P$. For $x\in P$, let $P_{\le x}=\{y\in P\mid y\le x\}$. The order complex $\Delta(P)$ is the geometric realization of an abstract simplicial complex whose simplices are finite chains $x_0<x_1<\cdots<x_n$ in $P$. Let $F\colon P\to\mathbf{Top}$ be a functor. We define two maps
\[
  f,g\colon\coprod_{x<y\in P}\Delta(P_{\le x})\times F(y)\to\coprod_{x\in P}\Delta(P_{\le x})\times F(x)
\]
by
\[
  f=\coprod_{x<y\in P}1_{\Delta(P_{\le x})}\times F(y>x)\quad\text{and}\quad g=\coprod_{x<y\in P}\iota_{x,y}\times 1_{F(y)},
\]
where $\iota_{x,y}\colon\Delta(P_{\le x})\to\Delta(P_{\le y})$ denotes the inclusion for $x<y$. The homotopy colimit $\hocolim F$ is defined to be the coequalizer of $f$ and $g$. By definition, there is a natural projection
\begin{equation}
  \label{projection}
  \pi\colon\hocolim F\to\Delta(P)
\end{equation}

We recall a property of regular CW complexes that we are going to use. For a CW complex $X$, let $P(X)$ denote its face poset. The following lemma is proved in \cite[Theorem 1.6, Chapter III]{LW}.

\begin{lemma}
  \label{subdivision}
  Let $X$ be a regular CW complex. Then there is a homeomorphism
  \[
    \Delta(P(X))\xrightarrow{\cong}X
  \]
  which restricts to a homeomorphism $\Delta(P(X)_{\le\sigma})\xrightarrow{\cong}\sigma$ for each face $\sigma$.
\end{lemma}

Now we describe $\Conf_r(X)$ in terms of a homotopy colimit. Similarly to the Fadell-Neuwirth fibration \cite{FN}, we consider the first projection $\pi\colon\Conf_r(X)\to X$. Then for each face $\sigma$ of $X$, we have
\[
  \pi^{-1}(\mathrm{Int}(\sigma))=\Conf_{r-1}(X(\sigma)).
\]
Thus since $X(\sigma)\subset X(\tau)$ for $\sigma>\tau$, $\Conf_r(X)$ is obtained by gluing $\sigma\times\Conf_{r-1}(X(\sigma))$ along the inclusions
\[
  \sigma\times\Conf_{r-1}(X(\sigma))\leftarrow\tau\times\Conf_{r-1}(X(\sigma))\to\tau\times\Conf_{r-1}(X(\tau))
\]
for $\sigma>\tau$. In other words, $\Conf_r(X)$ is homeomorphic to the coequalizer of two maps
\[
  f,g\colon\coprod_{\tau<\sigma\in P(X)}\tau\times\Conf_{r-1}(X(\sigma))\to\coprod_{\tau\in P(X)}\tau\times\Conf_{r-1}(X(\tau))
\]
defined by
\[
  f=\coprod_{\tau<\sigma\in P(X)}1_{\tau}\times\theta_{\sigma,\tau}\quad\text{and}\quad g=\coprod_{\tau<\sigma\in P(X)}\iota_{\tau,\sigma}\times 1_{\Conf_{r-1}(X(\sigma))},
\]
where $\theta_{\sigma,\tau}\colon\Conf_{r-1}(X(\sigma))\to\Conf_{r-1}(X(\tau))$ and $\iota_{\tau,\sigma}\colon\tau\to\sigma$ are inclusions for $\sigma>\tau$. Now we define a functor $F_r\colon P(X)\to\mathbf{Top}$
by
\[
  F_r(\sigma)=\Conf_{r-1}(X(\sigma))\quad\text{and}\quad F(\sigma>\tau)=\theta_{\sigma,\tau}.
\]
By Lemma \ref{subdivision}, there is a natural homeomorphism $\Delta(P(X)_{\le\sigma})\cong\sigma$ for each face $\sigma$ of $X$. Then by the above observation, we get:

\begin{theorem}
  \label{hocolim}
  There is a homeomorphism
  \[
    \Conf_r(X)\cong\hocolim F_r.
  \]
\end{theorem}

We construct a spectral sequence which computes the homology of a homotopy colimit and is essentially the same as the Bousfield-Kan spectral sequence. Let $P$ be a poset. By definition, we have
\begin{equation}
  \label{cell decomp}
  \Delta(P)=\bigcup_{x\in P}\Delta(P_{\le x})
\end{equation}
so that there is a filtration of $P$
\begin{equation}
  \label{filtration}
  P_0\subset\cdots\subset P_n\subset P_{n+1}\subset\cdots,
\end{equation}
where $P_n$ is the union of all $\Delta(P_{\le x})$ for $\dim\Delta(P_{\le x})\le n$. Let $F\colon P\to\mathbf{Top}$ be a functor. Then by applying the projection \eqref{projection}, we get a filtration of $\hocolim F$
\[
  \pi^{-1}(P_0)\subset\cdots\subset\pi^{-1}(P_n)\subset\pi^{-1}(P_{n+1})\subset\cdots.
\]
Suppose that $P=P(X)$ for a regular CW complex $X$. We describe the $E^1$-term of the spectral sequence associated with the above filtration. By Lemma \ref{subdivision}, \eqref{cell decomp} is identified with the cell structure of $X$, and so the filtration \eqref{filtration} is identified with the skeletal filtration of $X$. Then we get
\[
  E^1_{p,q}=H_{p+q}(\pi^{-1}(P(X)_p),\pi^{-1}(P(X)_{p-1}))\cong\bigoplus_{\substack{\sigma\in P(X)\\\dim\sigma=p}}H_q(F(\sigma)).
\]
Summarizing, we obtain:

\begin{proposition}
  \label{spectral seq}
  Let $X$ be a regular CW complex, and let $F\colon P(X)\to\mathbf{Top}$ be a functor. Then there is a spectral sequence
  \[
    E^1_{p,q}\cong\bigoplus_{\substack{\sigma\in P(X)\\\dim\sigma=p}}H_q(F(\sigma))\quad\Longrightarrow\quad H_{p+q}(\hocolim F).
  \]
\end{proposition}

We compute the acyclicity of $\Conf_r(X)$ by using the above spectral sequence.

\begin{lemma}
  \label{homology}
  Let $X$ be a regular CW complex, and let $F\colon P(X)\to\mathbf{Top}$ be a functor such that $F(\sigma)$ is $(n-\dim\sigma)$ acyclic for each $\sigma\in P(X)$ with $\dim\sigma\le n+1$. Then there is an isomorphism for $*\le n$
  \[
    H_*(\hocolim F)\cong H_*(X)
  \]
\end{lemma}

\begin{proof}
  We consider the spectral sequence of Proposition \ref{spectral seq}. Then by assumption on a functor $F$,
  \[
    E^1_{p,q}\cong
    \begin{cases}
      C_p&q=0\\
      0&q>0
    \end{cases}
    \quad\text{and}\quad E^1_{n+1,0}\cong\bigoplus_kC_{n+1}
  \]
  for $p+q\le n$, where $k\ge 1$ and $C_*$ denotes the cellular chain complex of $\Delta(P(X))$ associated with the cell decomposition \eqref{cell decomp}. Then by the construction of the spectral sequence, $d^1\colon E^1_{p,0}\to E^1_{p-1,0}$ is identified with the boundary map $\partial\colon C_p\to C_{p-1}$ for $p\le n$ and the sum of copies of the boundary map $\bigoplus_k\partial\colon\bigoplus_kC_{n+1}\to C_n$ for $p=n+1$. Clearly, we have
  \[
    \mathrm{Im}\left\{\bigoplus_k\partial\colon\bigoplus_kC_{n+1}\to C_n\right\}=\mathrm{Im}\{\partial\colon C_{n+1}\to C_n\}.
  \]
  Then we get
  \[
    E^2_{p,q}\cong
    \begin{cases}
      H_p(\Delta(P(X)))&q=0\\
      0&q>0
    \end{cases}
  \]
  for $p+q\le n$. Thus we obtain $E^2_{p,q}\cong E^\infty_{p,q}$ for $p+q\le n$, and the extension of $E_\infty^{p,q}$ to $H_*(\hocolim F)$ is trivial for $p+q\le n$. Therefore by Lemma \ref{subdivision}, the proof is complete.
\end{proof}

\begin{proposition}
  \label{Conf acyclicity}
  If $X$ is an $(r-1)$-complementary $n$-acyclic regular CW complex, then $\Conf_r(X)$ is $n$-acyclic.
\end{proposition}

\begin{proof}
  We induct on $r$. For $r=1$, $\Conf_1(X)=X$ is $n$-acyclic by assumption. Suppose that $\Conf_{r-1}(Y)$ is $n$-acyclic for any $(r-2)$-complementary $n$-acyclic regular CW complex $Y$. Since $X$ is $(r-1)$-complementary $n$-acyclic, $X(\sigma)$ is $(r-2)$-complementary $(n-\dim\sigma)$-acyclic for $\dim\sigma\le n+1$. Then $F_r(\sigma)=\Conf_{r-1}(X(\sigma))$ is $(n-\dim\sigma)$-acyclic for $\dim\sigma\le n+1$, and so by Theorem \ref{hocolim} and Lemma \ref{homology}, we obtain
  \[
    H_*(\Conf_r(X))\cong H_*(X)
  \]
  for $*\le n$. Thus since $X$ is $n$-acyclic, $\Conf_r(X)$ is $n$-acyclic, completing the induction.
\end{proof}

Now we are ready to prove Theorem \ref{main}.

\begin{proof}
  [Proof of Theorem \ref{main}]
  By Proposition \ref{Conf acyclicity}, $\Conf_r(X)$ is $(d(r-1)-1)$-acyclic. Then by Proposition \ref{criterion}, the proof is finished.
\end{proof}


\section{Atomicity}\label{atomicity}

Theorem \ref{main} shows that the Tverberg property is possessed not only by a simplex but also by a variety of CW complexes. But the Tverberg property of some CW complexes can be deduced from the that of other complexes. For example, as mentioned in Section \ref{introduction}, the Tverberg property of a polytopal sphere is deduced from a simplex. This section studies CW complexes having the Tverberg property that is not induced from other CW complexes.

We say that a regular CW complex $X$ is \emph{$(d,r)$-Tverberg} if for any continuous map $f\colon X\to\R^d$, there are pairwise disjoint faces $\sigma_1,\ldots,\sigma_r$ of $X$ such that $f(\sigma_1),\ldots,f(\sigma_r)$ have a point in common. For example, by Theorem \ref{main}, $(r-1)$-complementary $(d(r-1)-1)$-acyclic regular CW complexes are $(d,r)$-Tverberg. Let $X$ be a $(d,r)$-Tverberg regular CW complex. Observe that a regular CW complex $Y$ is $(d,r)$-Tverberg if either of the following conditions is satisfied:
\begin{enumerate}
  \item $X$ is a subcomplex of $Y$;

  \item $Y$ is a refinement of $X$, that is, $X\cong Y$ and each face of $X$ is the union of faces of $Y$.
\end{enumerate}
This observation leads us to:

\begin{definition}
  A $(d,r)$-Tverberg regular CW complex is called \emph{atomic} if it does not include a proper subcomplex which is $(d,r)$-Tverberg or it is not a refinement of a proper $(d,r)$-Tverberg complex.
\end{definition}

Here is a fundamental problem on $(d,r)$-Tverberg complexes.

\begin{problem}
  \label{atomic number}
  Given $d,r$ and $n$, are there only finitely many atomic $(d,r)$-Tverberg finite complexes of dimension $n$?
\end{problem}

First, we consider 1-dimensional $(1,2)$-Tverberg finite complexes. Let $C_n$ denote the cycle graph with $n$ vertices for $n\ge 3$. Then by Corollary \ref{Tverberg sphere}, $C_n$ is $(1,2)$-Tverberg. Let $Y$ be the $Y$-shaped graph depicted below. Then by the intermediate value theorem, we can see that $Y$ is $(1,2)$-Tverberg too.

\begin{figure}[htbp]
  \centering
  \begin{tikzpicture}[x=0.8cm, y=0.8cm, thick]
    \draw (0,0)--(0.86,0.5);
    \draw (0,0)--(-0.86,0.5);
    \draw (0,0)--(0,-1);
    \fill (0,0) circle [radius=2pt];
    \fill (0.86,0.5) circle [radius=2pt];
    \fill (-0.86,0.5) circle [radius=2pt];
    \fill (0,-1) circle [radius=2pt];
  \end{tikzpicture}
\end{figure}

\begin{proposition}
  \label{atomic (1,2)}
  The only atomic 1-dimensional $(1,2)$-Tverberg finite complexes are $C_3$ and $Y$.
\end{proposition}

\begin{proof}
  Clearly, any proper subcomplex of $C_3$ and $Y$ is not $(1,2)$-Tverberg, and $C_3$ and $Y$ are not refinements of other regular CW complexes. Then $C_3$ and $Y$ are atomic $(1,2)$-Tverberg finite complex of dimension 1. Let $X$ be a $(1,2)$-Tverberg finite complex of dimension 1. Suppose $X$ is not a forest. Then it includes $C_n$ for some $n\ge 2$. If $X$ includes $C_n$ for some $n\ge 3$, then it is $(1,2)$-Tverberg. If $X$ includes only $C_2$, then $X$ is a disjoint union of finitely many path graphs with multiple edges. So $X$ is not $(1,2)$-Tverberg. Suppose $X$ is a forest. If $X$ does not include $Y$, then it is a disjoint union of finitely many path graphs, and so it is not $(1,2)$-Tverberg. Thus $X$ must include $Y$, completing the proof.
\end{proof}

\begin{remark}
  If we remove its center vertex of $Y$, then it becomes disconnected. Hence $Y$ is not 1-complementary 0-acyclic, so that we cannot apply Theorem \ref{main} for $d=1$ and $r=2$ to deduce that $Y$ is $(1,2)$-Tverberg. However, we can directly see $\wgt_{\Z/2}(\Conf_2(Y))=1$, implying that $Y$ is $(1,2)$-Tverberg, because $\Conf_2(Y)$ is a hexagon so that Lemma \ref{weight} applies.
\end{remark}

Next, we consider $(2,2)$-Tverberg polyhedral 2-spheres. Let $\partial\Delta^n$ denote the boundary of an $n$-simplex.

\begin{proposition}
  The only atomic $(2,2)$-Tverberg polyhedral 2-sphere is $\partial\Delta^3$.
\end{proposition}

\begin{proof}
  By Steinitz's theorem, every polyhedral 2-sphere is polytopal, and as in \cite[p. 200]{G}, every polytopal sphere is a refinement of the boundary of a simplex. On the other hand, $\partial\Delta^3$ is $(2,2)$-Tverberg by the topological Tverberg theorem, and each proper subcomplex of $\partial\Delta^3$ is not $(2,2)$-Tverberg. Thus the proof is done.
\end{proof}

The 2-sphere in Figure \ref{non-polyhedral} is an atomic $(2,2)$-Tverberg complex, and there may be other atomic $(2,2)$-Tverberg 2-spheres which are not polyhedral. Then we pose a problem much weaker than Problem \ref{atomic number} but still interesting.

\begin{problem}
  Are there only finitely many atomic $(2,2)$-Tverberg 2-spheres?
\end{problem}

\end{document}